\documentclass[11pt]{amsart}
\usepackage{amsfonts}
\usepackage{amsmath}
\usepackage{amsthm}
\usepackage{amssymb}
\usepackage{mathrsfs}
\usepackage[numbers]{natbib}
\usepackage[fit]{truncate}
\usepackage{hyperref}
\usepackage{theoremref}

\usepackage{xcolor}
\usepackage{tikz-cd}
\usepackage{thmtools}
\usepackage{thm-restate}
\usepackage{fmtcount}
\usepackage{enumerate}	
\usepackage[UKenglish]{babel}
\usepackage[UKenglish]{isodate}
\usepackage{verbatim}
\usepackage[utf8]{inputenc}
\usepackage{dsfont}
\usepackage{etoolbox}
\makeatletter
\patchcmd{\@maketitle}
{\ifx\@empty\@dedicatory}
{\ifx\@empty\@date \else {\vskip3ex \centering\footnotesize\@date\par\vskip1ex}\fi
	\ifx\@empty\@dedicatory}
{}{}
\patchcmd{\@adminfootnotes}
{\ifx\@empty\@date\else \@footnotetext{\@setdate}\fi}
{}{}{}
\makeatother

\makeatletter
\def\namedlabel#1#2{\begingroup
	#2%
	\def\@currentlabel{#2}%
	\phantomsection\label{#1}\endgroup
}
\makeatother

\numberwithin{equation}{section}

 \newtheoremstyle{dotless}{}{}{\itshape}{}{\bfseries}{}{ }{}

\theoremstyle{dotless}

\MakeRobust{\ref}

\makeatletter
\newcommand{\labeltext}[2]{%
	\@bsphack
	\csname phantomsection\endcsname 
	\def\@currentlabel{#1}{\label{#2}}%
	\@esphack
}
\makeatother

\newenvironment{acknowledgements}{\ 
	
	{\textsl{Acknowledgements.}}}{}

\makeatletter
\def\blfootnote{\gdef\@thefnmark{}\@footnotetext}
\makeatother

\newcommand{\N}[0]{\mathbb{N}}
\newcommand{\Z}[0]{\mathbb{Z}}
\newcommand{\Q}[0]{\mathbb{Q}}

\newcommand{\supp}[0]{\mathrm{supp}}

\newcommand{\brackets}[1]{\left( #1 \right)}
\newcommand{\setbr}[1]{\left\{ #1 \right\}}
\newcommand{\pow}[1]{\!\left(\!\left( #1 \right)\!\right)}

\newcommand\restr[2]{{
		\left.\kern-\nulldelimiterspace 
		#1
		\vphantom{\big|} 
		\right|_{#2}
}}
\newcommand{\vmin}[0]{v_{\min}}

\DeclareMathOperator{\Char}{char}

\newcommand{\bbF}{{\mathbb F}}

\newcommand{\bbK}{{\mathbb K}}
\newcommand{\K}{\mathbb K}

\newcommand{\bbZ}{{\mathbb Z}}

\newcommand{\cF}{{\mathcal F}}

\newcommand{\cP}{{\mathcal P}}

\newcommand{\cW}{{\mathcal W}}

\newcommand{\inv}{^{-1}}

\renewcommand{\epsilon}{\varepsilon}
\renewcommand{\phi}{\varphi}
\renewcommand{\theta}{\vartheta}

\usepackage{theoremref}

\theoremstyle{definition}
\newtheorem{defn}{Definition}[section]
\newtheorem{conditions}[defn]{Conditions}
\theoremstyle{plain}
\newtheorem{lemma}[defn]{Lemma}
\theoremstyle{plain}
\newtheorem*{lemma*}{Lemma}
\theoremstyle{plain}
\newtheorem{prop}[defn]{Proposition}
\theoremstyle{plain}
\theoremstyle{plain}
\newtheorem*{prop*}{Proposition}
\newtheorem{theorem}[defn]{Theorem}
\theoremstyle{plain}
\newtheorem*{teorema*}{Theorem}
\theoremstyle{plain}
\newtheorem{cor}[defn]{Corollary}
\theoremstyle{plain}
\newtheorem*{cor*}{Corollary}
\theoremstyle{definition}
\newtheorem{es}[defn]{Example}
\theoremstyle{definition}

\theoremstyle{plain}

\theoremstyle{definition}

\theoremstyle{definition}
\newtheorem{rmk}[defn]{Remark}
\theoremstyle{plain}

\theoremstyle{definition}
\newtheorem{notation}[defn]{Notation}

{\left\lbrace\begin{array}{@{}l@{}}}%
	{\end{array}\right.}

\newcommand{\vertiii}[1]{{\left\vert\kern-0.25ex\left\vert\kern-0.25ex\left\vert #1 
		\right\vert\kern-0.25ex\right\vert\kern-0.25ex\right\vert}}

\newcounter{nootje}
\begin{document}

\begin{abstract}
	In this note, we study substructures of generalised power series fields induced by families of well-ordered subsets of the group of exponents. We characterise the  set-theoretic and algebraic properties of the induced substructures in terms of conditions on the families.
	We extend the work of Rayner by giving both \emph{necessary} and sufficient conditions to obtain truncation closed subgroups, subrings and subfields.
\end{abstract}

	\title[On Rayner Structures]{On Rayner Structures}
	
	\author[L.~S.~Krapp]{Lothar Sebastian Krapp}
	\author[S.~Kuhlmann]{Salma Kuhlmann}
	\author[M.~Serra]{Michele Serra}
	
	\address{Fachbereich Mathematik und Statistik, Universität Konstanz, 78457 Konstanz, Germany}
	\email{sebastian.krapp@uni-konstanz.de}
	
	\address{Fachbereich Mathematik und Statistik, Universität Konstanz, 78457 Konstanz, Germany}
	\email{salma.kuhlmann@uni-konstanz.de}
	
	\address{Fachbereich Mathematik und Statistik, Universität Konstanz, 78457 Konstanz, Germany}
	\email{michele.serra@uni-konstanz.de}
	
	\date{\today}
	
	\cleanlookdateon
	
	\maketitle
	
	\blfootnote{\textup{2020} \textit{Mathematics Subject Classification}: 13J05 (12J20 16W60 06F20)
		}
	
\section{Introduction}
	In his work \cite{rayner}, Rayner presented a construction method for subfields of generalised power series fields induced by families of well-ordered subsets of the value group. More specifically, for a field $k$, an ordered abelian group $G$ and a family $\cF$ of well-ordered subsets of $G$, Rayner introduced sufficient conditions on $\cF$ in order that the $k$-hull of $\cF$ (\thref{def:khull})
	forms a subfield of $k\pow{G}$. Here $k\pow{G}$ denotes the field of generalised power series with coefficient field $k$ and value group $G$.
	
	In this note, we study \emph{necessary} and sufficient conditions on the field $k$, the group $G$ and the family $\cF$, for the $k$-hull of $\cF$ to satisfy certain properties. By a careful analysis of these conditions, we characterise when the $k$-hull of $\cF$ is a subgroup (\thref{prop:group}), a subring (\thref{prop:ring}) and a subfield (\thref{prop:field2}) of $k\pow{G}$.
	Among these, Hahn fields (\thref{def:hahn}) are of special interest. In particular, Mourgues and Ressayre \cite{mourgues} studied the interesting class of truncation closed Hahn fields. In \thref{cor:hahn}, we characterise $k$-hulls that are truncation closed Hahn fields.
	Finally, we show in \thref{thm:main} that the $k$-hull of $\cF$ is a Rayner field (\thref{def:rayner}) if and only if it is a Hahn field.

\section{Preliminaries, terminology and notations}

	\textbf{Throughout this note, let $k$ be a field, let $G$ be an additive ordered abelian group and let $\cF$ be a family of well-ordered subsets of $G$.}
	For $A,B\subseteq G$ and $g\in G$ we denote by $\cW(A)$ the family of well-ordered subsets of $A$, by $\langle A\rangle$ the subgroup of $G$ generated by $A$, by $A\oplus B$ the set of sums $\{a+b\mid a\in A,b\in B\}$, by $A+g$ the translation set $\{a+g\mid a\in A\}$ and by 
	$\bigoplus_{n\in \N}A$ the set of finite sums of elements of $A$, i.e.\ $\bigoplus_{n\in \N}A=\setbr{\sum_{i=1}^na_i\mid n\in \N,  a_1,\ldots,a_n\in A}$ (where $\N=\{0,1,\ldots\}$). By convention, $\bigoplus_{n\in \N}\emptyset = \{0\}$.
	
	On $\cF$ we consider the following set theoretic (S) and algebraic (A) conditions.
	
	\begin{conditions}\thlabel{conditions}
		\begin{description}
			\item[\namedlabel{singletons}{(S1)}]
			$g\in G$ implies $\{g\}\in\cF $;
			\item[ \namedlabel{subsets}{(S2)}]
			$A\in \cF$ and $B\subseteq A$ implies $B\in \cF$;
			\item[ \namedlabel{union}{(S3)}]
			$A,B\in \cF$ implies $A\cup B \in \cF$;
			\item[ \namedlabel{zero}{(S4)}]
			$ \{0\} \in\cF$;
			\item[ \namedlabel{nonempty}{(S5)}]
			$ \cF \neq \emptyset $;
			\item[ \namedlabel{initial}{(S6)}]
			if $A\in \cF$ and $B$ is an initial segment of $ A$, then $B\in \cF$;
			\item[ \namedlabel{generated-group}{(A1)}]
			$\left\langle \bigcup_{A\in\cF}A\right\rangle=G$;
			\item[ \namedlabel{sums}{(A2)}]
			$ A,B\in\cF$ implies $A\oplus B\in\cF $;
			\item[ \namedlabel{translations}{(A3)}]
			$A \in \cF$ and $g \in G$ implies $A+g \in \cF$;
			\item[ \namedlabel{semigroup}{(A4)}]
			$A\in \cF$ and $A\subseteq G^{\geq0}$ implies $ \bigoplus_{n\in\N}A \in \cF$;
			\item[ \namedlabel{negatives}{(A5)}]
			if $g \in G$ such that $\{g\}\in \cF$, then $\{-g\}\in \cF$.
		\end{description}
	\end{conditions}

	\begin{rmk}
	\begin{enumerate}[(i)]
		\item
		Explicit implications are \ref{singletons} $ \Rightarrow $ \ref{zero} $ \Rightarrow $ \ref{nonempty}, \ref{subsets} $ \Rightarrow $ \ref{initial}, \ref{singletons} $ \Rightarrow $ \ref{generated-group} and \ref{singletons} $ \Rightarrow $ \ref{negatives}.
		\item 
		Other implications are implicit. See, for example, \thref{lemma:equivalences}.
	\end{enumerate}
	\end{rmk}

	We denote the generalised power series field with coefficient field $k$ and value group $G$ by $\K$. It consists of the set of all functions $s\colon G \to k$ whose \textbf{support} $\supp(s)=\{g\in G\mid s(g)\neq 0\}$ is a well-ordered subset of $G$. For any $g\in G$, we denote by $t^g$ the characteristic function mapping $g$ to $1$ and everything else to $0$, and we call $t^g$ a (monic) \textbf{monomial} of $\K$. This way, we can express a power series $s\in \K$ by $s=\sum_{g\in G}s_gt^g$, where $s_g=s(g)\in k$.
	For any power series $r,s\in \K$, their sum is given by $r+s=\sum_{g\in G}(r_g+s_g)t^g$ and their product by $rs=\sum_{g\in G}c_gt^g$ with $c_g=\sum_{h\in G}r_hs_{g-h}$. 
	Note that $ \supp(r+s)\subseteq \supp(r)\cup\supp(s) $ and $ \supp(rs)\subseteq\supp(r)\oplus\supp(s) $.
	These operations make $\K$ a field (cf.~\cite{hahn,neumann}). 
	
	\begin{defn}\thlabel{def:hahn}
		We call $ \K $ the \textbf{maximal Hahn field} with coefficient field $k$ and value group $G$. A subfield $K$ of $\K$ with $\{\alpha t^g\mid \alpha\in k, g\in G\}\subseteq K$ is called a \textbf{Hahn field} in $\K$. 
	\end{defn}

	Now we introduce the subsets of $\K$ induced by $\cF$ that we are going to study.

	\begin{defn}\thlabel{def:khull}
			We call the set $$k\pow{\cF} =\{ a\in \K \mid  \supp(a)\in\cF \} \subseteq \K$$ the $ k $\textbf{-hull} of $ \cF $ in $\K$.
	\end{defn}

	\begin{rmk}\thlabel{rmk:unitary-ring}
		Note that $ k\pow{\cF} $ contains the coefficient field $k$ if and only if $ \emptyset\in \cF$ and $\{0\}\in\cF $.
	\end{rmk}

	\begin{notation}\thlabel{not:wellorder}
		Whenever the family $ \cF $ is of the form $ \cW(S) $ for some set $ S \subseteq G$, we write $ k\pow{S} $ instead of $ k\pow{\cF} $.
	\end{notation}

	\begin{es}\thlabel{es:khulls}
		\begin{enumerate}[(i)]
			\item\label{es:khulls:1} In the valuation theoretic study of maximal Hahn fields, $k$-hulls play an important role, as they give rise to the valuation ring and its maximal ideal: for the standard valuation $\vmin\colon \K\setminus\setbr{0}\to G, s \mapsto \min \supp (s)$ on $\K$, the valuation ring is given by $k\pow{G^{\geq 0}}$ and its maximal ideal by $k\pow{G^{> 0}}$. 
			
			\item\label{es:khulls:2} Let $\kappa$ be an uncountable regular cardinal and let $\cF$ be the family of all well-ordered subsets of $G$ of cardinality less than $\kappa$.
			Then $k\pow{\cF}$ is the Hahn field of $\kappa$-bounded power series. Such Hahn fields provide natural constructions for models of real exponentiation (cf.\ \cite{kuhlmann,kuhlmannmatusinskimantova}).
		\end{enumerate}
	\end{es}
	
	We lastly introduce the notions of restriction and truncation closure for $k$-hulls.
	Due to  the work of Mourgues and Ressayre \cite{mourgues}, truncation closed subfields of maximal Hahn fields are of particular interest in the study of integer parts of ordered fields and have been the subject of study ever since (cf.\ e.g.\ \cite{biljakovic,fornasiero}). 
	
	\begin{defn}\thlabel{def:trunc-rest-closed}
		The $k$-hull $ k\pow{\cF} $ of $\cF$ is called \textbf{restriction closed} if $ \cF $ satisfies \ref{subsets}. It is called \textbf{truncation closed} if $ \cF $ satisfies \ref{initial}.
	\end{defn}

\section{Properties of $k$-hulls}

	We start by summarising the sufficient conditions on $\cF$ given in \cite{rayner} in order to ensure that $k\pow{\cF}$ has certain algebraic properties as the following theorem (cf.\ \cite[page~147]{rayner}).
	
	\begin{theorem}\thlabel{thm:raynermain}
		\begin{enumerate}[(i)]
			\item\label{thm:raynermain:1} If $\cF$ satisfies \ref{subsets}, \ref{union} and \ref{nonempty}, then $k\pow{\cF}$ is a subgroup of $(\K,+)$.
			
			\item If $\cF$ satisfies \ref{subsets}, \ref{union}, \ref{nonempty}, \ref{translations} and \ref{semigroup}, then $k\pow{\cF}$ is a subring (with identity) of $\K$.
			
			\item\label{thm:raynermain:3} If $\cF$ satisfies \ref{subsets}, \ref{union}, \ref{nonempty},   \ref{generated-group}, \ref{translations} and \ref{semigroup}, then $k\pow{\cF}$ is a subfield of $\K$.
		\end{enumerate}
	\end{theorem}
	
	\thref{thm:raynermain}~\eqref{thm:raynermain:3} gives rise to the following definition.
		
	\begin{defn}\thlabel{def:rayner}
		We call $\cF$ a \textbf{Rayner field family} in $G$ if it satisfies conditions
		\ref{subsets}, \ref{union}, \ref{nonempty},   \ref{generated-group}, \ref{translations} and \ref{semigroup}. 
		If $\cF$ is a Rayner field family in $G$, then we call the field $k\pow{\cF}$ a \textbf{Rayner field} (with coefficient field $k$ and field family $\cF$).
	\end{defn}

	\begin{rmk}\thlabel{rmk:trivialgroup}
		Rayner does actually not include \ref{nonempty} in his definition of a field family. However, if $G=\{0\}$, then the empty family would satisfy \ref{subsets}, \ref{union}, \ref{generated-group}, \ref{translations} and \ref{semigroup} but its $k$-hull would be the empty set. In fact by our definition, if $G=\{0\}$, then the only Rayner field family in $G$ is $\cF=\{\emptyset,\{0\}\}$.
	\end{rmk}

	Since Rayner is merely interested in sufficient conditions on $\cF$ in order to ensure that $k\pow{\cF}$ exhibits certain algebraic properties, some of the conditions he poses may not be necessary. In the following, we carefully analyse further the relations between the \thref{conditions} and the properties of $ k\pow{\cF} $ as an algebraic substructure of $\K$.

	\begin{prop}\thlabel{prop:group}
		Suppose that $k\neq \bbF_2$. Then $ k\pow{\cF} $ is a subgroup of $ (\K,+) $ if and only if $ \cF $ satisfies \ref{subsets}, \ref{union} and \ref{nonempty}.
	\end{prop}
	\begin{proof}
			\thref{thm:raynermain}~\eqref{thm:raynermain:1} gives us the backward direction. Conversely, suppose that $ k\pow{\cF} $ is an additive subgroup of $\K$. Then it contains $0$, whence $ \emptyset=\supp(0)\in\cF$. This establishes \ref{nonempty}. Now let $ A\in \cF $ and let $ B\subseteq A $. Let $ a \in \bbK $ be such that $ \supp(a) = A $ and let $ c\in \K $ be defined by
		\[
		\begin{cases}
		c_g\notin \{0,-a_g\}, &g\in B;\quad \text{(which is possible as $ k\neq\bbF_2 $)}\\
		c_g= -a_g, &g\in A\setminus B;\\
		c_g = 0, & g\in G\setminus A.
		\end{cases}
		\]
		Then $ \supp( c) = A $, whence $ c\in k\pow{\cF} $. Since $ k\pow{\cF} $ is a group, we have $ a+c \in k\pow{\cF} $ and thus $B =  \supp(a+c) \in \mathcal{F}$, yielding \ref{subsets}. Now let $ A,B\in\cF $ and let $ a\in \K $ be such that $ \supp(a) = A $. Then choose $ b\in\K $ such that $ \supp(b)= B $ and $ b_g\neq -a_g $ for every $ g\in A\cap B $ (this is always possible since $ k\neq\bbF_2 $). This yields $ \supp(a+b)=A\cup B\in\cF $ and thus establishes \ref{union}.
	\end{proof}

	Since \ref{subsets} implies \ref{initial}, we obtain the following.
		
	\begin{cor}\thlabel{cor:group-restr-cl}
		{ Suppose that $k\neq \bbF_2$.} If $ k\pow{\cF} $ is an additive group, then it is restriction closed (and thus, in particular, truncation closed).\qed
	\end{cor}
	
	We now show that the conclusion of \thref{prop:group} fails for $k=\bbF_2$. 
	{ Note that for any family $\mathcal{F}$, there is a bijective correspondence between   $\mathcal{F}$ and $\bbF_2(\!(\mathcal{F})\!)$ given by $A \mapsto \sum_{g\in A}t^g$.
		\begin{es}\thlabel{ex:counterexf2}
			Let $\bbF_2(s)$ be the subfield of $\bbF_2(\!(\Z)\!)$ generated by $s= t^2+t^3$. We show that $\bbF_2(s)$ does not contain $t^2$ and is thus not truncation closed. It suffices to prove that for any $p,q\in \bbF_2[X]$ with $q\neq 0$ we have $t^2q(s)\neq p(s)$. We do so by induction on the degree of $p$.
			
			Clearly, if $\deg(p)=0$, then for any $q \in \bbF_2[X]$ we have $t^2q(s)\neq p(s)$. Let $n\in \N$ and suppose that the claim holds for any polynomial of degree $n-1$. Let $p(X)=X^n+\sum_{i=0}^{n-1}a_iX^i$ and assume, for a contradiction, that for some $q\in \bbF_2[X]\setminus\setbr{0}$ we have $t^2q(s)= p(s)$. Let $q(X)=\sum_{j=0}^mb_jX^j$. Then
			$$\sum_{j=0}^mb_jt^2(t^2+t^3)^j=t^2q(s)=p(s)=(t^2+t^3)^n+\sum_{i=0}^{n-1}a_i(t^2+t^3)^i.$$
			Comparing coefficients of $t^0$, we obtain $a_0=0$ and thus $$\sum_{j=0}^mb_jt^2(t^2+t^3)^j=(t^2+t^3)\brackets{(t^2+t^3)^{n-1}+\sum_{i=1}^{n-1}a_i(t^2+t^3)^{i-1}}\!.$$
			Comparing coefficients of $t^3$, we obtain
			$a_1=0$, whence
			$$\sum_{j=0}^mb_jt^2(t^2+t^3)^j=(t^2+t^3)^2\brackets{(t^2+t^3)^{n-2}+\sum_{i=2}^{n-1}a_i(t^2+t^3)^{i-2}}\!.$$
			Finally, comparing coefficients of $t^2$, we obtain
			$b_0=0$.
			Hence,
			$$t^2\sum_{j=0}^{m-1}b_{j+1}s^j=s\brackets{s^{n-2}+\sum_{i=2}^{n-1}a_is^{i-2}}\!.$$
			This shows that for the polynomial $p'(X)=X\brackets{X^{n-2}+\sum_{i=2}^{n-1}a_iX^{i-2}}$ of degree $n-1$ there exists $q'(X)$ such that $t^2q'(s)=p'(s)$ giving us the required contradiction.
			Now let $\mathcal{F}$ be the set of all supports of elements of $\bbF_2(s)$. Then $\bbF_2(s)=\bbF_2\pow{\mathcal{F}}$ and $\bbF_2\pow{\mathcal{F}}$ is a subfield of $\bbF_2\pow{\Z}$ which is not truncation closed.
		\end{es}
		
		We now also consider multiplication on $k\pow{\cF}$. 
		
		\begin{lemma}\thlabel{lem:ring}
			If $ \cF $ satisfies \ref{subsets}, \ref{union}, \ref{nonempty} and \ref{sums}, then $k\pow{\cF}$ is a subring (possibly without identity) of $ \K $.
		\end{lemma}
		
		\begin{proof}
			By \thref{prop:group}, $ k\pow{\cF} $ is an additive subgroup of $ (\K,+) $.
			Now let $ a,b\in k\pow{\cF} $. We set $ A=\supp(a)$, $ B=\supp(b)$ and let $ c=ab\in\K $. Then by definition of the product, we have $ \supp(ab)\subseteq A\oplus B \in \cF $. Hence, by \ref{sums} we obtain $ \supp(ab)\in\cF $ and thus $ ab\in k\pow{\cF} $.
		\end{proof}
	}

	The converse of \thref{lem:ring} does not hold unconditionally, as \thref{ex:counterexf2} presents a $k$-hull that is a subfield of $\bbF_2\pow{\bbZ}$ not satisfying \ref{subsets}. However, under certain conditions on the characteristic or the cardinality of $k$, also the converse of \thref{lem:ring} holds, as \thref{prop:ring2} will show.
	
	\begin{lemma}\thlabel{lem:multsupp}
			Let $a,b\in \bbK$ and suppose that at least one of the following conditions holds:
		\begin{enumerate}[(i)]
			\item\label{lem:multsupp:1}
			There exists a totally ordered subfield $ Q\subseteq k $ such that $a_g,b_g\in Q^{\geq 0}$ for any $g\in G$.
			\item The set $\{a_g\mid g\in \supp(a)\}\cup \{b_g\mid g\in \supp(b)\}$ is algebraically independent over the prime field of $k$.
		\end{enumerate}
			Then $\supp(ab)=\supp(a)\oplus\supp(b)$.
			In particular, condition \eqref{lem:multsupp:1} holds if $ \Char(k) = 0 $ and $a_g,b_g\in \Q^{\geq 0}$ for any $g\in G$.
	\end{lemma}

	\begin{proof}
		Since $\supp(cd)\subseteq \supp(c)\oplus\supp(d)$ holds for any $c,d\in \bbK$, we only have to verify $A\oplus B\subseteq \supp(ab)$ for $A=\supp(a)$ and $B=\supp(b)$.
		Let $h\in A\oplus B$ with $h=h_1+h_2$ for some $h_1\in A$ and $h_2\in B$. Then
		\begin{align}(ab)_h=\sum_{\substack{(g_1,g_2)\in (A\times B)\\g_1+g_2=h}}a_{g_1}b_{g_2}=a_{h_1}b_{h_2}+\sum_{\substack{(g_1,g_2)\in (A\times B)\setminus\{(h_1,h_2)\}\\g_1+g_2=h}}a_{g_1}b_{g_2}.\label{eq:productsupp}\end{align}
		\begin{enumerate}[(i)]
			\item
			Suppose that for any $g\in G$ we have $a_g,b_g\in Q^{\geq 0}$. Then $a_{h_1},b_{h_2}>0$ and thus \eqref{eq:productsupp} shows that $(ab)_h>0$.
			\item Suppose that $\{a_g\mid g\in \supp(a)\}\cup \{b_g\mid g\in \supp(b)\}$ is algebraically independent over the prime field $P$ of $k$. By \eqref{eq:productsupp}, $(ab)_h= p(a_{r_1},\ldots,a_{r_i},b_{s_1},\ldots,b_{s_j})$ for some $i,j\in \N^{>0}$, some non-zero polynomial  $p\in P[X_1,\ldots,\allowbreak X_i,Y_1,\ldots,Y_j]$ and some ${r_1},\ldots,{r_i}\in A,{s_1},\ldots,{s_j}\in B$. By algebraic independence, we obtain $(ab)_h\neq 0$.
		\end{enumerate}
		In both cases, we obtain $(ab)_h\neq 0$. Hence, $h\in \supp(ab)$, as required.
	\end{proof}
	
	\begin{prop}\thlabel{prop:ring}\thlabel{prop:ring2}
		Suppose that at least one of the following conditions is satisfied:
		\begin{enumerate}[(i)]
			\item\label{char_0} 	$\mathrm{char}(k) = 0$;
			\item\label{card_k} 	$ |k|\geq \sup\setbr{|A| \mid A\in \cW(G)} $.

		\end{enumerate}
		\noindent
		Then $ k\pow{\cF} $ is a subring (possibly without identity) of $ \K $ if and only if $ \cF $ satisfies conditions \ref{subsets}, \ref{union}, \ref{nonempty} and \ref{sums}.
	\end{prop}
	\begin{proof}
		{ 
			The backward direction follows from \thref{lem:ring}, independently of conditions \eqref{char_0} or \eqref{card_k}. 
			For the converse, suppose that $k\pow{\cF}$ is a subring of $\K$.
			
			First assume that $ \Char(k) = 0 $. By \thref{prop:group}, it remains to verify \ref{sums}. Let $A,B \in \cF$ and set $a=\sum_{g\in A}t^g$ and $b=\sum_{g\in B}t^g$. Then by \thref{lem:multsupp}, we obtain that $A\oplus B=\supp(ab)\in \cF$.
			
			Now assume that \eqref{card_k} holds and that $\Char(k)=p>0$. Note that $k$ is infinite, since any ordered abelian group contains a well-ordered subset of cardinality $\aleph_0$.
			Hence, again by \thref{prop:group}, it remains to verify \ref{sums}.
			We distinguish two cases.
			
			\textbf{Case 1:} Suppose that $|k|>\aleph_0$. 
			Let $A,B \in \cF$. Then $|k|\geq |A|+|B|$. Since $|k|$ is uncountable, its absolute transcendence degree over its prime field $P$ is equal to $|k|$. Hence, there exists a set $C=\{a_g\mid g\in A\}\cup \{b_h\mid h\in B\}\subseteq k$ (where the $a_g$ and $b_h$ are pairwise distinct) of cardinality at most $|k|$ such that $C$ is algebraically independent over $P$. We set $a=\sum_{g\in A}a_gt^{g}$ and $b=\sum_{h\in B}b_ht^{h}$.
			Then by \thref{lem:multsupp}, we obtain that $A\oplus B=\supp(ab)\in \cF$.
		
			\textbf{Case 2:} Suppose that $|k|=\aleph_0$. Let $A,B\in \cF$. As $|A|\leq |k|=\aleph_0$, the set $A$ is countable. Since the prime field of $k$ is the finite field $\bbF_p$, the degree of $k$ over $\bbF_p$ is $\aleph_0$. We can thus choose a set $ \setbr{a_g \mid g\in A}\subseteq k $ linearly independent over $ \bbF_p $ and set $a=\sum_{g\in A}a_gt^g$.  Moreover, we set $b=\sum_{g\in B}t^g$. Let $h_1\in A$, $h_2\in B$ and set $h:=h_1+h_2$. We compute
			$$(ab)_{h}=a_{h_1}+\sum_{\substack{(g_1,g_2)\in (A\times B)\setminus\{(h_1,h_2)\}\\g_1+g_2=h}}a_{g_1}.$$
			Since this is a non-trivial linear combination of elements from $ \setbr{a_g \mid g\in A} $, we obtain $(ab)_h\neq 0$. This shows $h\in \supp(ab)$, establishing $A\oplus B = \supp(ab)\in \cF$, as required.
		}
	\end{proof}

	\begin{rmk}\thlabel{rmk:conditionii}
		\begin{enumerate}[(i)]
			\item Conditions \eqref{char_0} and \eqref{card_k} in \thref{prop:ring} ensure that in its proof the sums of the coefficients of the power series representing $a$ and $b$ do not cancel in the product $ab$, whence $\supp(ab)=\supp(a)\oplus\supp(b)$. 
			
			\item\label{rmk:conditionii:2} Condition \eqref{card_k} is always satisfied if $|k|\geq |G|$, as $|G|\geq |A|$ for any $A \in \cW(G)$.
		\end{enumerate}
		
	\end{rmk}
	
	We now consider the field structure on $k$-hulls.
	
	\begin{lemma}\thlabel{prop:field}
		\begin{enumerate}[(i)]
			\item \label{prop:field:1}
			Suppose that $\cF$  satisfies conditions  \ref{subsets}, \ref{union}, \ref{zero}, \ref{sums}, \ref{semigroup} and \ref{negatives}. Then $k\pow{\cF}$ is a subfield of $\K$.
			\item\label{prop:field:2}
			If $\cF$ satisfies conditions \ref{singletons}, \ref{subsets}, \ref{union}, \ref{sums}, \ref{semigroup} then $ K $ is a Hahn field. 
		\end{enumerate}
	\end{lemma}
	\begin{proof}
	\begin{enumerate}[(i)]
		\item 	\thref{lem:ring} and \thref{rmk:unitary-ring} imply that $ k\pow{\cF} $ is a ring with identity.
		Let $b\in k\pow{\cF}\setminus\setbr{0}$ be arbitrary and let $h=\min\supp(b)$.
		Then by \ref{subsets} and \ref{negatives}, we have $t^{-h}\in k\pow{\cF}$ and thus obtain $$ b_h^{-1}t^{-h} = 1 + \sum_{g \in G^{ >h}} b_h^{-1} b_g t^{g-h} \in k(\!(\mathcal{F})\!).$$ Now set $a = - \sum_{g \in G^{ >h}} b_h^{-1} b_g t^{g-h}$ and let $A=\supp(a)$; so  $ A\subseteq G^{>0}$ and  $A \in \cF $. Then 
		$(1-a)^{-1}=\sum_{i=0}^\infty a^i$ (cf.\ Neumann's Lemma~\cite[page~211]{neumann}).
		Hence, $ \supp(1-a)\inv \subseteq \bigoplus_{n\in \N}A$ and, by \ref{semigroup} and \ref{subsets}, it lies in $ \cF $. 
		This implies $b_ht^hb^{-1}=\brackets{b_h^{-1}t^{-h}b}^{-1}\in k\pow{\cF}$, whence $b^{-1}\in k\pow{\cF}$, as required.
		\item 
		Condition \ref{singletons} implies \ref{zero}, so part (i) gives $ k\pow{\cF} \subseteq \K $. Moreover, \ref{singletons} implies $ k(t^g\mid g\in G)\subseteq k\pow{\cF} $. So $ k\pow{\cF} $ is a Hahn field.
	\end{enumerate}
	\end{proof}

\begin{rmk}\label{rmk3.12}
	Let $g_0 \in G$ and fix $b = \sum_{g \geq g_0} b_g t^g \in \mathbb{K}^\times$ with $b_{g_0} \neq 0$. Then we have
	\[
	b\inv = b_{g_0}\inv t^{-g_0}\brackets{1-\sum_{g>g_0}(-b_{g_0})^{-1}b_gt^{g-g_0}}\inv
	\]
	Let $ \epsilon= \sum_{g>g_0}(-b_{g_0})^{-1}b_gt^{g-g_0} $. 
	Then $ \supp(\epsilon) = \brackets{\supp(b)-g_0}\setminus\setbr{0} $.
	Since $ \supp\!\brackets{(1-\epsilon)\inv}\subseteq \bigoplus_{n\in \N}\supp(\epsilon) $ (see the proof of \thref{prop:field}), it follows that 
	\begin{align*}
		\supp(b\inv) &= 
	\supp\brackets{b_{g_0}\inv t^{-g_0}(1-\epsilon)\inv} \subseteq \brackets{\bigoplus_{n\in\N}\supp(\epsilon)}-g_0  \\
	&=  \brackets{\bigoplus_{n\in\N}\brackets{\brackets{\supp(b)-g_0}\setminus\setbr{0}}}-g_0.
	\end{align*}
	
\end{rmk}

	\begin{lemma}\thlabel{lemma:neumann}
		Let $a\in k\pow{G^{>0}}$ and suppose that 
			there exists a totally ordered subfield $ Q\subseteq k $ such that $a_g\in Q^{>0}$ for any $g\in \supp(a)$. 
		Then $\supp((1-a)\inv)=\bigoplus_{n\in \N}\supp(a)$.
		In particular, 
		this holds if $ \Char(k) = 0 $ and $a_g\in \Q^{>0}$ for any $g\in \supp(a)$.
	\end{lemma}

	\begin{proof}
		Set $\supp(a)=A$.
		We first show that for any $n\in \N^{>0}$ we have $$\supp(a^n) = nA:=\underbrace{A\oplus\ldots\oplus A}_{n\text{ times}}.$$
		Since for any $c,d\in \bbK$ we have $\supp(cd)\subseteq \supp(c)\oplus\supp(d)$, it is clear that $\supp(a^n) \subseteq n A$. For the converse, let
		$h\in nA$ and fix $h_1,\ldots,h_n\in A$ such that 
		$h=h_1+\ldots+h_n$. 
		Then 
		$$(a^n)_h={\sum_{\substack{(g_1,\ldots,g_n)\in A^n\\g_1+\ldots+g_n=h} }a_{g_1}\ldots a_{g_n}}.$$
		We can re-write this sum to obtain
		\begin{align}(a^n)_h= {a_{h_1}\ldots a_{h_n}}+{\sum_{\substack{(g_1,\ldots,g_n)\in A^n\setminus\{(h_1,\ldots,h_n)\}\\g_1+\ldots+g_n=h} }a_{g_1}\ldots a_{g_n}}.\label{eq:powexpansion}\end{align}

		We now have that in \eqref{eq:powexpansion} both ${a_{h_1}\ldots a_{h_n}}>0$ and $a_{g_1}\ldots a_{g_n}>0$ for all terms of the sum. Hence, $(a^n)_h\neq 0$.
	
		Now by Neumann's Lemma we have
		$(1-a)\inv=\sum_{n=0}^\infty a^n$. 
		We now show that 
		$$\supp((1-a)^{-1})=\bigcup_{n=1}^\infty nA\cup \{0\}.$$
		Note that the right-hand side is equal to $\bigoplus_{n\in \N}A$.
		The inclusion
				$$\supp((1-a)^{-1})\subseteq\bigcup_{n=1}^\infty nA\cup \{0\}$$
		is clear, as $\supp(a^n)=nA$ for any $n\in \N^{>0}$. 
		
		For the converse, first note that $0\in  \supp((1-a)^{-1})$, as $(1-a)\inv = 1+\sum_{n=1}^\infty a^n$.
		Now fix $n\in \N^{>0}$, let $h\in nA$ and fix $h_1,\ldots,h_n\in A$ such that 
		$h=h_1+\ldots+h_n$. 
		Arguing as above, we can write $((1-a)\inv)_h$ as follows:
	\begin{align}((1-a)\inv )_h&= {a_{h_1}\ldots a_{h_n}}+{\sum_{\substack{(g_1,\ldots,g_n)\in A^n\setminus\{(h_1,\ldots,h_n)\}\\g_1+\ldots+g_n=h} }a_{g_1}\ldots a_{g_n}}+\notag\\&+\sum_{j\in \N\setminus\{n\}}\brackets{\sum_{\substack{(g_1,\ldots,g_j)\in A^j\\g_1+\ldots+g_j=h} }a_{g_1}\ldots a_{g_j}}.\label{eq:powexpansion2}\end{align}
	The previous arguments can also be applied to show
	$${\sum_{\substack{(g_1,\ldots,g_n)\in A^n\setminus\{(h_1,\ldots,h_n)\}\\g_1+\ldots+g_n=h} }a_{g_1}\ldots a_{g_n}}+\sum_{j\in \N\setminus\{n\}}\brackets{\sum_{\substack{(g_1,\ldots,g_j)\in A^j\\g_1+\ldots+g_j=h} }a_{g_1}\ldots a_{g_j}}\geq 0$$
	in \eqref{eq:powexpansion2}. Hence, we also obtain $((1-a)\inv )_h\neq 0$, as required.
	\end{proof}

\begin{es}
	Let $ p\in\N $ be prime and let $ a = t-t^p \in \bbF_p\pow{\Z} $.
	Then $\bigoplus_{n\in \N}\supp(a) = \N$.
	We will show that $ \supp((1-a)\inv)\subsetneq \bigoplus_{n\in \N}\supp(a)$.
	By expressing $ a = t(1-t^{p-1}) $, we obtain
	\[b:=(1-a)\inv=\sum_{i=0}^\infty t^i(1-t^{p-1})^i=\sum_{i=0}^\infty\sum_{j=0}^i {i \choose j} (-1)^jt^{i+(p-1)j}.\]
	A direct calculation gives us 
	$$b_p=(-1)^1\cdot {1 \choose 1}+(-1)^0\cdot {p\choose 0}=0.$$
	Thus, we obtain $p\notin \supp((1-a)\inv)$, whence $\supp((1-a)\inv)\neq \N$.
\end{es}
		
	\begin{prop}\thlabel{prop:field2}\thlabel{cor:hahn}

		Suppose that 
			$\mathrm{char}(k) = 0$. Then the following hold:
		\begin{enumerate}[(i)]
			\item 
			$k\pow{\cF}$ is a subfield of $\K$ if and only if $\cF$ satisfies conditions   \ref{subsets}, \ref{union}, \ref{zero}, \ref{sums}, \ref{semigroup}  and \ref{negatives}.
			\item 
			$k\pow{\cF}$ is a Hahn field in $\K$ if and only if $\cF$ satisfies conditions \ref{singletons}, \ref{subsets}, \ref{union}, \ref{sums}, \ref{semigroup}. Hence, if $ k\pow{\cF} $ is a Hahn field in $\K$, then it is restriction closed and thus also truncation closed.
		\end{enumerate}
	\end{prop}
	
	\begin{proof}
		\begin{enumerate}[(i)]
			\item By \thref{prop:field}, only the forward direction needs to be shown.
		Let $ k\pow{\cF} $ be a field. Then \thref{prop:ring} and \thref{rmk:unitary-ring} imply that $ \cF $ satisfies conditions \ref{zero}, \ref{subsets}, \ref{union} and \ref{sums}. To prove condition \ref{semigroup} let $ A\in\cF $ be such that $ A\subseteq G^{\geq0} $ and let $ a = \sum_{g\in A^{>0}}t^g $. Then $ \supp (1-a) = A\cup\{0\} $. By \thref{lemma:neumann}, the support of $(1-a)\inv$ is ${\bigoplus_{n\in \N}(A\cup\{0\})}$.
		Since $ k\pow{\cF} $ is a field, we have $ (1-a)\inv\in k\pow{\cF} $ and thus $\bigoplus_{n\in \N}(A\cup\{0\})\in \cF$. We obtain by \ref{subsets} that $\bigoplus_{n\in \N}A\in \cF$, establishing \ref{semigroup}. Finally, \ref{negatives} follows easily, as for any monomial $t^g\in k\pow{\cF}$ we already have $t^{-g}\in k\pow{\cF}$.
		\item
		Follows from part (i) and \thref{prop:field}~\eqref{prop:field:2}.
		\end{enumerate}
	\end{proof}
	
	In the last proposition we obtained necessary and sufficient conditions (in the case $\mathrm{char}(k)=0$) in order that $k\pow{\cF}$ is a Hahn field.

	Finally, we show that $k\pow{\cF}$ is a Hahn field if and only if it is a Rayner field. By \thref{cor:hahn}, it suffices to show that $\cF$ is a Rayner field family if and only if it satisfies \ref{singletons}, \ref{subsets}, \ref{union}, \ref{sums}, \ref{semigroup}.
	We first prove that if $G$ is non-trivial, then condition \ref{generated-group} in \thref{def:rayner} can be removed.
	
	\begin{lemma}\thlabel{lemma:equivalences}
		Suppose that $G\neq \{0\}$ and that $ \cF $ satisfies conditions \ref{subsets}, 
		\ref{semigroup} and \ref{translations}. 
		Then conditions \ref{nonempty}, \ref{zero}, \ref{singletons} and \ref{generated-group}
		are equivalent.
	\end{lemma}
	\begin{proof}

		$ \ref{nonempty}\Rightarrow\ref{zero} $: Let $ \cF\neq\emptyset $ and let $ A\in \cF $. If $A\neq \emptyset$, then for any $g\in A$, we obtain by \ref{subsets} and \ref{translations} that $\{0\}=\{g\}-g\in \cF$. Otherwise, by \ref{semigroup} we have $\bigoplus_{n\in\N}A=\{0\}\in \cF$.
		$ \ref{zero}\Rightarrow\ref{singletons} $: This follows immediately from \ref{translations}.
		$ \ref{singletons}\Rightarrow\ref{generated-group} $ and $\ref{generated-group}\Rightarrow\ref{nonempty}$ are obvious. Note that for the latter we need that $G\neq \{0\}$.
	\end{proof}
	
	\begin{lemma}\thlabel{lemma:premain}
		If $ k\pow{\cF} $ is a Rayner field then it is a Hahn field.
	\end{lemma}
\begin{proof}
Suppose that $ k\pow{\cF} $ is a Rayner field, that is, $ \cF $ satisfies  
\ref{subsets}, 
\ref{union},
\ref{nonempty},
\ref{generated-group}, 
\ref{translations} and
\ref{semigroup}
. By \thref{prop:field}, it remains to verify \ref{singletons} and \ref{sums}.
If $G=\{0\}$, then by \thref{rmk:trivialgroup} we have $\cF=\{\emptyset,\{0\}\}$, which trivially satisfies \ref{singletons} and \ref{sums}.
If $G\neq \{0\}$, then \thref{lemma:equivalences} shows that $\cF$ satisfies \ref{singletons}. We thus only have to show \ref{sums}.
Let $A,B\in \cF$ be non-empty.
Let $m_A=\min A$ and $m_B=\min B$. 
Then by \ref{translations}, we have $A-m_A,B-m_B\in \cF$. 
Note that $A-m_A,B-m_B\in G^{\geq0}$. 
Hence, by \ref{union} and \ref{semigroup}, we obtain 
$$
\bigoplus_{n\in \N}((A-m_A)\cup (B-m_B))\in \cF.
$$
In particular, $(A-m_A)\oplus(B-m_B)\in \cF$.
By \ref{translations} we obtain $A\oplus B = ((A-m_A)\oplus(B-m_B))+(m_A+m_B)\in \cF$.
\end{proof}

\noindent
For $ \Char(k) = 0 $ also the converse of \thref{lemma:premain} holds.

\begin{theorem}\thlabel{thm:main}
	Suppose that $\mathrm{char}(k)=0$. Then $ k\pow{\cF} $ is a Rayner field if and only if it is a Hahn field.
\end{theorem}
	\begin{proof}
		The forward direction is a special case of \thref{lemma:premain}.
		Suppose now that $ k\pow{\cF}$ is a Hahn field. By \thref{prop:field2}, $ \cF $ satisfies \ref{singletons}, \ref{subsets}, \ref{union}, \ref{sums}, \ref{semigroup}. 
		We need to show that \ref{nonempty}, \ref{generated-group} and \ref{translations} hold. Again, if $G=\{0\}$, then $\cF=\{\emptyset,\{0\}\}$ and there is nothing to prove. 
		Otherwise, by  \thref{lemma:equivalences} it suffices to show \ref{translations}. 
		Let $A\in \cF$ and let $g\in G$. 
		Then by \ref{singletons}, we have $\{g\}\in \cF$. 
		Hence, by \ref{sums}, we obtain $A+g=A\oplus \{g\}\in \cF$, as required.
	\end{proof}
	
	\noindent
The following is a counterexample to the conclusion of \thref{thm:main} for $ \Char(k) =2 $.
	\begin{es}		
		Let $G$ be a countable ordered abelian group. Moreover, let $K=\bbF_2(t^g\mid g\in G)$. Note that since $\bbF_2$ and $G$ are countable, also $K$ is countable. Moreover, $K$ is a Hahn field in $\bbF_2\pow{G}$ with $K=\bbF_2\pow{\cF}$ for $\cF=\{\supp(a)\mid a\in K\}$. We show that $K$ does not satisfy \ref{subsets} and is thus not a Rayner field.
		
		First note that for any $g\in G^{>0}$ we have
		$b:=(1-t^g)^{-1}=\sum_{i=0}^\infty t^{ig}$. Hence $\supp(b)=\{ig\mid i\in \N\}$, i.e.\ $b$ has countably infinite support. Since $|\cF|\leq |K|$, also $\cF$ is countable. However, the powerset $\cP(\supp(b))$ is uncountable, whence $\cP(\supp(b))\setminus \cF\neq \emptyset$. This shows that $K$ does not have the property \ref{subsets}.
	\end{es}

\noindent
By \thref{thm:main} and \thref{prop:field2}, we immediately obtain the following.

\begin{cor}
	Suppose that $\mathrm{char}(k)=0$. Then $k\pow{\cF}$ is a Rayner field (or, equivalently, a Hahn field) if and only if  $\cF$ satisfies all of \thref{conditions}.\qed
\end{cor}
		
	\begin{acknowledgements} 
		We thank Mickaël Matusinski for several helpful comments and corrections.	
	\end{acknowledgements}
	
	\begin{footnotesize}
		
	\end{footnotesize}	
	
\end{document}